\documentclass{birkjour}
\usepackage{cite}
\usepackage{hyperref}
\hypersetup{
	colorlinks,
	linkcolor = {blue},
	citecolor = {blue},
	filecolor = {blue},
	urlcolor = {blue} 
} 

\usepackage{xcolor}

\usepackage{url}
\usepackage{tkz-euclide}
\usepackage[autostyle, english = american]{csquotes}
\MakeOuterQuote{"}

 \newtheorem{thm}{Theorem}[section]
 
 \newtheorem{lem}[thm]{Lemma}
 \newtheorem{prop}[thm]{Proposition}
 \theoremstyle{definition}
 \newtheorem{defn}[thm]{Definition}
 \theoremstyle{remark}
 \newtheorem{rem}[thm]{Remark}
 
 \numberwithin{equation}{section}

 \definecolor{DarkNavy}{RGB}{0,200,255}  
\definecolor{VibrantMagenta}{RGB}{255,0,127}  
\definecolor{LightGray}{RGB}{211,211,211}  
\definecolor{LightGray2}{RGB}{221,221,221}  
\definecolor{newGreen}{RGB}{85, 214, 69}  
\definecolor{LightBlue}{RGB}{173,225,255} 
\definecolor{lPurple}{RGB}{245, 190, 215}
\definecolor{neonPurple}{RGB}{177, 0, 253}  
\definecolor{uuuuuu}{rgb}{0,0,0}

\begin{document}

\title[Cycling along Euler road]
 {Cycling along Euler road}

\author[Dylan Wyrzykowski]{Dylan Wyrzykowski}

\address{
College of Inter-Faculty Individual Studies in Mathematics and Natural Sciences \\
University of Warsaw\\
Stefana Banacha 2C\\
02-097 Warsaw\\
Poland
}

\email{dylanwyrzykowski@gmail.com}

\subjclass{Primary 51M05; Secondary 51M04, 51N20}

\keywords{Euler line, cyclic polygon, multidimensional polygon, homothety, Shinagawa coefficients, triangle center.}

\begin{abstract} We introduce the notion of $P_{\lambda}$ points, which canonically parametrize points on the Euler line. This allows us to show that the Euler line of any $d$-dimensional inscribed polygon in Euclidean space arises from the Euler lines of its sub-polygons, beginning from the Euler line of a point in the plane. Furthermore, we situate $P_{\lambda}$ points in the literature of modern triangle centers.
\end{abstract}

\maketitle

\section{Introduction}

For a given triangle, its centroid, orthocenter, circumcenter, and nine-point center are collinear on what is now called the Euler line. In this paper, we study the Euler line for multidimensional polygons\footnote{We use this term to indicate we do not concern ourselves with the structure of this set of points, simply its cyclic order. } inscribed in hyperspheres.

Beginning in the plane, modern literature \cite{Gomez} considers the following notion. Let $A_{1}A_{2}A_{3}$, with $A_i = (x_i, y_i)$, be a triangle on the unit circle. The orthocenter of this triangle is given by
\begin{equation}
    H = \left( \sum_{i=1}^3 x_i, \sum_{i=1}^{3} y_i \right).
\end{equation}
It can then be observed that for a cyclic quadrilateral $A_1A_2A_3A_4$ (lying on the unit circle), the orthocenters of each sub-triangle $A_1A_2A_3A_4 \setminus A_i$ lie on the circle given by
\begin{equation}
    \left( \sum_{i=1}^{4} x_i - x \right)^2 + \left( \sum_{i=1}^{4} y_i - y \right)^2 = 1.
\end{equation}
Consequently, Gómez \cite{Gomez} defines the point $(x_1+x_2+x_3+x_4, y_1+y_2+y_3+y_4)$ as the orthocenter of the quadrilateral $A_1A_2A_3A_4$. This structure readily generalizes to any cyclic $n$-gon by defining its orthocenter as the sum of the coordinates of its vertices. Specifically, the orthocenters of each sub-polygon $A_1A_2\ldots A_n \setminus A_i$ lie on a circle of radius $1$ centered at the orthocenter of $A_1A_2\ldots A_n$. The line passing through the circumcenter of $A_1A_2\ldots A_n$ and this newly defined orthocenter is then the Euler line of the cyclic polygon. It is straightforward to show that the centroid of this polygon, $(\frac{1}{n}\sum x_i, \frac{1}{n} \sum y_i )$, also lies on this line. Similar ideas have been proposed in \cite{Yaglom,Andrica}, with slight differences in definition from \cite{Collings}. 

Building on these ideas, we introduce the notion of $P_{\lambda}$ points, which generalizes the inductive construction described above to $d$-dimensional Euclidean space. We demonstrate that each "layer" of this construction is connected through a homothety, and that lower-dimensional results can be viewed as "shadows" of higher-dimensional structures. Finally, we show that $P_{\lambda}$ points have long been present in the literature of triangle center geometry, appearing in foundational works such as \cite{Kimberling}. Thus, we generalize results about these triangle centers given in \cite{Collings, Gomez,Yaglom,Andrica, Rabinowitz2025, Bradley,Mammana}. Our generalization of the Euler line differs from others in the literature \cite{Buba,Myakishev,Snapper,Tabachnikov}, primarily through our definition of the centers themselves. 

Throughout this text, we denote the position vector of a point by its corresponding lowercase letter. We let $\mathbb{S}^{d-1}(Q,r) \subset \mathbb{R}^d$ be a $(d-1)$-dimensional hypersphere with center $Q$ and radius $r \geq 0$. The notation $\Phi_{X}^{\mu}$ will denote a homothety centered at a point $X$ with ratio $\mu \in \mathbb{R}$.

\section{$P_{\lambda}$ points and the Euler line.}

Let $\sigma = A_1A_2\ldots A_n$ be a multidimensional $n$-polygon in $\mathbb{R}^d$ inscribed in the unit hypersphere centered at the origin, $O$. Its circumcenter is therefore the origin, $O = 0 \cdot \sum_{i} a_i$, and its centroid is given by $G = \frac{1}{n} \sum_{i} a_i$. With this in mind, we propose to define the following object that bundles both of these into one.

\begin{defn}[$P_{\lambda}$ points] Let $A_1, A_2, \ldots, A_n \in \mathbb{S}^{d-1}(O,1)$, with $n \in \mathbb{Z}^+, d > 1$, then the point $P_{\lambda}$ such that \begin{equation}
    p_{\lambda} = \lambda \sum_{j=1}^{n} a_j, \quad \lambda \in \mathbb{R},
\end{equation}
    is called the $P_{\lambda}$ point of $A_1A_2\ldots A_n$. 
\end{defn}

Under this definition, specific values of $\lambda$ correspond to the classical centers of $A_1A_2\ldots A_n$: $P_{\frac{1}{n}} = G$ is the centroid, $P_{0}=O$ is the circumcenter, and $P_{1}=H$ is the quasi-orthocenter or orthocenter\footnote{If $\sigma$ is not orthocentric, one can define the notion of a "Monge Point" (see \cite{Buba}) in the simplex. This point coincides with $P_1$ when $\sigma$ is orthocentric; the extension of this notion beyond the simplex has not been studied, hence we call $P_1$ the quasi-orthocenter.} if $\sigma$ is orthocentric, or the generalized nine-point center $P_{\frac{1}{2}} = E$. We therefore define the set $ \ell = \{ P_{\lambda} : \lambda \in \mathbb{R} \}$ as the Euler line of $\sigma$. This definition is further motivated by the following theorem, which describes the behavior of the $P_{\lambda}$ family of points and reveals how these Euler lines are constructed inductively (Fig. 1).

\begin{figure}[ht]
    \begin{center}
        \begin{tikzpicture}\label{mainpi}
            \draw [line width = 0.5pt] (0,0) circle (4cm);
            \draw [line width = 0.5pt] (-0.44638501553674725,3.9750145179488663)-- (-3.954536412057831,0.6013665818024623);
            \draw [line width = 0.5pt] (-3.954536412057831,0.6013665818024623)-- (-1.391670223971025,-3.750100530347478);
            \draw [line width = 0.5pt] (-1.391670223971025,-3.750100530347478)-- (3.1433662974655063,-2.4737114463813263);
            \draw [line width = 0.5pt] (3.1433662974655063,-2.4737114463813263)-- (3.5755653496631368,1.7931348053864578);
            \draw [line width = 0.5pt] (3.5755653496631368,1.7931348053864578)-- (-0.44638501553674725,3.9750145179488663);
            \draw [line width = 0.5pt, color=black, fill = newGreen, opacity = 0.5] (1.6269588025402908,-0.15188755113116004) circle (1.333333333333333cm);
            \draw [line width = 0.5pt, color=black, fill = red, opacity = 0.5] (-0.8830751180333657,-0.5491436256591585) circle (1.3333333333333335cm);
            \draw [line width = 0.5pt, color=black, fill = orange, opacity = 0.5] (0.4575750036999291,-1.2764368631799612) circle (1.333333333333333cm);
            \draw [line width = 0.5pt, color=black, fill = DarkNavy, opacity = 0.5] (-0.739008767300822,0.873138458263436) circle (1.333333333333333cm);
            \draw [line width = 0.5pt, color=black, fill = yellow, opacity = 0.5] (0.7726700731780217,1.2986014862521533) circle (1.333333333333333cm);
            \draw [line width = 0.5pt,domain=-5.086692292917039:4.80212263964246] plot(\x,{(-0--0.04856797613632723*\x)/0.30877999852101345});
            \draw [line width = 1pt, color=black] (0.30877999852101307,0.04856797613632715) circle (1.333333333333333cm);

        \draw [line width = 0.5pt, dashed] (-0.44638501553674725,3.9750145179488663) -- (0.4575750036999291,-1.2764368631799612);
        \draw [line width = 0.5pt, dashed] (3.5755653496631368,1.7931348053864578) -- (-0.8830751180333657,-0.5491436256591585);
        \draw [line width = 0.5pt, dashed] (3.1433662974655063,-2.4737114463813263) -- (-0.739008767300822,0.873138458263436);
        \draw [line width = 0.5pt, dashed] (-1.391670223971025,-3.750100530347478) -- (0.7726700731780217,1.2986014862521533);
        \draw [line width = 0.5pt, dashed] (-3.954536412057831,0.6013665818024623) -- (1.6269588025402908,-0.15188755113116004);
        
        \draw [fill=blue] (0.23,0.03) circle (1pt);
        \draw[color=blue, yshift = 15] (0.15,-0.2) node {$T$};
            
            \draw [fill=black] (-0.44638501553674725,3.9750145179488663) circle (1pt);
            \draw[color=black] (-0.31770999055997384,4.273571087675104) node {$A_1$};
            \draw [fill=black] (3.5755653496631368,1.7931348053864578) circle (1pt);
            \draw[color=black, right] (3.7040252933647424,2.0845253002224067) node {$A_2$};
            \draw [fill=black] (3.1433662974655063,-2.4737114463813263) circle (1pt);
            \draw[color=black, right] (3.2458529192467367,-2.4971984409576575) node {$A_3$};
            \draw [fill=black] (-1.391670223971025,-3.750100530347478) circle (1pt);
            \draw[color=black, below] (-1.6540460817374902,-3.7953535009586754) node {$A_4$};
            \draw [fill=black] (-3.954536412057831,0.6013665818024623) circle (1pt);
            \draw[color=black, left] (-4.135813108210021,1.0027294168882246) node {$A_5$};
            
            \draw [fill=black] (0,0) circle (1pt);
            \draw[color=black, left, xshift = -5] (0.09591784718544788,0.22638178296604716) node {$O$};
            
            \draw [fill=neonPurple] (0.7726700731780217,1.2986014862521533) circle (1pt);
\draw[color=neonPurple] (0.9040830070880411,1.5627178741435657) node {$P_{\lambda, 4}^{4}$};
\draw [fill=neonPurple] (1.6269588025402908,-0.15188755113116004) circle (1pt);
\draw[color=neonPurple, right] (1.7567927033632185,-0.02815842487728973) node {$P_{\lambda, 5}^{4}$};
\draw [fill=neonPurple] (0.4575750036999291,-1.2764368631799612) circle (1pt);
\draw[color=neonPurple,below] (0.4459106329700355,-1.3390404952704749) node {$P_{\lambda, 1}^{4}$};
\draw [fill=neonPurple] (-0.8830751180333657,-0.5491436256591585) circle (1pt);
\draw[color=neonPurple, left, yshift = -4] (-1.1067846348743169,-0.4990578093874629) node {$P_{\lambda, 2}^{4}$};
\draw [fill=neonPurple] (-0.739008767300822,0.873138458263436) circle (1pt);
\draw[color=neonPurple] (-0.9286064893839815,1.1809075623785603) node {$P_{\lambda, 3}^{4}$};
\draw [fill=neonPurple] (0.30877999852101345,0.04856797613632723) circle (1pt);
\draw[color=neonPurple, right, yshift = 4] (0.40136609659745165,0.2772898245347145) node {$P_{\lambda}^5$};
            
            \draw [line width = 0.5pt, color=neonPurple](0.7726700731780217,1.2986014862521533) -- (1.6269588025402908,-0.15188755113116004) -- (0.4575750036999291,-1.2764368631799612) -- (-0.8830751180333657,-0.5491436256591585)-- (-0.739008767300822,0.873138458263436) -- cycle;
        \end{tikzpicture}
    \end{center}
    \caption{Theorem \ref{themG} for $d=2, n=5$.}
\end{figure}
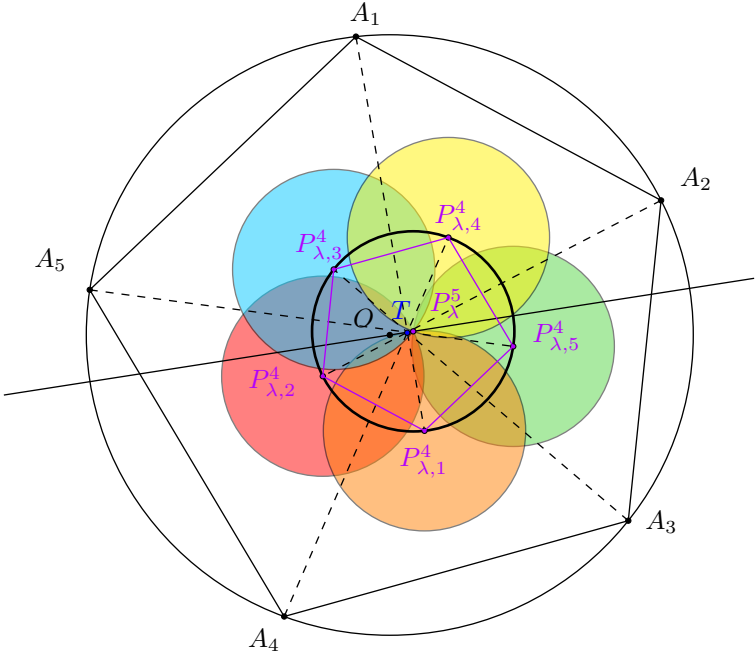

\begin{thm}\label{themG}
    Let $A_1, A_2, \ldots, A_n \in \mathbb{S}^{d-1}(O,1)$, with $n > 2, d > 1$. For each $i \in \{1,2,\ldots,n\}$, let $P_{\lambda, i}^{n-1}$ denote the $P_{\lambda}$ point of the sub-polygon $A_1A_2\ldots A_n \setminus A_i$, and let $P_{\lambda}^n$ be the $P_{\lambda}$ point of $A_1A_2\ldots A_n$. Then, the following holds:
    \begin{enumerate}
        \item the multidimensional polygon $P_{\lambda,1}^{n-1}P_{\lambda,2}^{n-1}\ldots P_{\lambda,n}^{n-1}$ is inscribed in the hypersphere $\mathbb{S}^{d-1}(P_{\lambda}^n, |\lambda|)$,
        \item each hypersphere $\mathbb{S}^{d-1}(P_{\lambda, i}^{n-1},|\lambda|)$ passes through $P_{\lambda}^{n}$,
        \item there exists a homothety $\Phi_{\mathbf{T}}^{-\lambda}$ that maps the multidimensional polygon $A_1A_2\ldots A_n$ to the multidimensional polygon $P_{\lambda,1}^{n-1}P_{\lambda,2}^{n-1}\ldots P_{\lambda,n}^{n-1}$,
        \item the lines $A_iP_{\lambda,i}^{n-1}$ concur at a point on the Euler line of $A_1A_2\ldots A_n$.
    \end{enumerate}
\end{thm}

\begin{proof}
    For point 1, we compute the distance between $P_{\lambda}^n$ and each $P_{\lambda,i}^{n-1}$:
\begin{equation}
    \begin{split}
        |p_{\lambda}^n - p_{\lambda,i}^{n-1}| &= \left| \lambda \sum_{j=1}^{n} a_j - \left( \lambda \sum_{j=1}^{n} a_j - \lambda a_i \right)\right| \\
        &= |\lambda a_i| = |\lambda|,
    \end{split}
\end{equation}
    since $A_i \in \mathbb{S}^{d-1}(O,1)$. This proves that all points $P_{\lambda,i}^{n-1}$ lie on the hypersphere $\mathbb{S}^{d-1}(P_{\lambda}^n, |\lambda|)$. For point 2, by symmetry, $P_{\lambda}^n$ lies at distance $|\lambda|$ from each $P_{\lambda,i}^{n-1}$, placing it on each hypersphere $\mathbb{S}^{d-1}(P_{\lambda,i}^{n-1}, |\lambda|)$. For point 3, consider the homothety $\Phi_{\mathbf{T}}^{-\lambda} : \mathbb{R}^d \to \mathbb{R}^d$ defined as follows:
    \begin{equation}\label{homodef}
        \Phi_{\mathbf{T}}^{-\lambda}: x \mapsto -\lambda(x - \mathbf{t}) + \mathbf{t}, \quad \text{where} \quad \mathbf{t} = \frac{\lambda}{\lambda + 1} \sum_{j=1}^{n} a_j.
    \end{equation}
    
    Thus, for any $i \in \{1,2,\ldots,n\}$ we have: \begin{equation}
        \begin{split}
              \Phi^{-\lambda}_{\mathbf{T}}(a_i) &= -\lambda\left(a_i - \frac{\lambda}{\lambda + 1}\sum_{j=1}^{n}a_j\right) + \frac{\lambda}{\lambda + 1}\sum_{j=1}^{n}a_j \\
        &= -\lambda a_i + \frac{\lambda^2}{\lambda + 1}\sum_{j=1}^{n}a_j + \frac{\lambda}{\lambda + 1}\sum_{j=1}^{n}a_j \\
        &= -\lambda a_i + \lambda\left(\frac{\lambda + 1}{\lambda + 1}\sum_{j=1}^{n}a_j\right) \\
        &= \lambda\sum_{j=1}^{n}a_j - \lambda a_i = p_{\lambda,i}^{n-1}.
        \end{split}
    \end{equation}
     Proving point 3. For point 4, since $\Phi_{\mathbf{T}}^{-\lambda}$ maps $A_i$ to $P_{\lambda,i}^{n-1}$, the lines $A_iP_{\lambda,i}^{n-1}$ must all pass through the center of homothety $\mathbf{T}$, which lies on the Euler line of $A_1A_2\ldots A_n$ since $\mathbf{t}= P_{\frac{\lambda}{\lambda +1}}$, by definition \eqref{homodef}.
\end{proof}

\begin{rem}
    We note that this framework implicitly defines the Euler line of a single point $A$ as the set $\{ \lambda a : \lambda \in \mathbb{R} \}$. In this sense, the Euler line of a point can be said to \textit{sprout}, generating the Euler line of any inscribed multidimensional polygon\footnote{A Geogebra animation presenting this idea for the planar case can be found at: \url{https://www.geogebra.org/m/jtzew4dz}. One interacts with the animation by clicking the $n$-slider and going forward or backward.}.
\end{rem}

\begin{figure}[ht]
    \begin{center}

\begin{tikzpicture}\label{niner}
    [line cap=round,line join=round,>=triangle 45,x=1cm,y=1cm]
\clip(-5.229893417463077,-3.027386489087619) rectangle (5.996951647346085,4.892408885515603);
\draw [line width = 0.5pt] (0,0) circle (3cm);
\draw [line width = 0.5pt] (-1.8493882694807926,2.362152202697117)-- (-2.830657392238432,-0.9936693251610013);
\draw [line width = 0.5pt] (-2.830657392238432,-0.9936693251610013)-- (2.9768080563973762,-0.3723087366255009);
\draw [line width = 0.5pt] (2.9768080563973762,-0.3723087366255009)-- (1.0119275667902548,2.824181757530835);
\draw [line width = 0.5pt] (1.0119275667902548,2.824181757530835)-- (-1.8493882694807926,2.362152202697117);
\draw [line width = 0.5pt] (-1.8493882694807926,2.362152202697117)-- (2.9768080563973762,-0.3723087366255009);
\draw [line width = 0.5pt] (-2.830657392238432,-0.9936693251610013)-- (1.0119275667902548,2.824181757530835);
\draw [line width = 0.5pt, fill = orange, opacity = 0.5] (1.0696736768534196,2.4070126118012256) circle (1.5cm);
\draw [line width = 0.5pt, fill = newGreen, opacity = 0.5] (0.5790391154745993,0.7291018478721661) circle (1.5cm);
\draw [line width = 0.5pt, fill = red, opacity = 0.5] (-0.8516188026609242,0.4980870704553072) circle (1.5cm);
\draw [line width = 0.5pt, fill = DarkNavy, opacity = 0.5] (-1.834059047464485,2.096332317533475) circle (1.5cm);
\draw [fill=black] (-1.8493882694807926,2.362152202697117) circle(1pt);
\draw[color=black] (-2.0512728414398924,2.671655452772132) node {$A_1$};
\draw [fill=black] (-2.830657392238432,-0.9936693251610013) circle(1pt);
\draw[color=black] (-2.993086345446762,-1.127705841801035) node {$A_2$};
\draw [fill=black] (2.9768080563973762,-0.3723087366255009) circle(1pt);
\draw[color=black] (3.192915987689268,-0.5176675494329491) node {$A_3$};
\draw [fill=black] (1.0119275667902548,2.824181757530835) circle(1pt);
\draw[color=black] (1.095240456037604,3.056942795320397) node {$A_4$};
\draw [fill=uuuuuu] (0.07307533207947214,-0.6829890308932511) circle (1pt);
\draw[color=uuuuuu, yshift = -5] (0.12131967348504552,-0.8066330563441478) node {$M_{23}$};
\draw [fill=uuuuuu] (-2.3400228308596125,0.6842414387680578) circle (1pt);
\draw[color=uuuuuu, yshift =-4] (-2.7576329694450443,0.8308381494859778) node {$M_{12}$};
\draw [fill=uuuuuu] (0.5637098934582918,0.994921733035808) circle (1pt);
\draw[color=uuuuuu, xshift = 3] (0.7099531134893391,1.237530344398035) node {$M_{13}$};
\draw [fill=uuuuuu] (-0.4187303513452689,2.593166980113976) circle (1pt);
\draw[color=uuuuuu, yshift = -3] (-0.08202642397098314,2.939216107319538) node {$M_{14}$};
\draw [fill=uuuuuu] (-0.9093649127240886,0.9152562161849167) circle (1pt);
\draw[color=uuuuuu, xshift = -17, yshift =6] (-0.6278501592476917,0.884350280395459) node {$M_{24}$};
\draw [fill=uuuuuu] (1.9943678115938155,1.225936510452667) circle (1pt);
\draw[color=uuuuuu, yshift = -6, xshift = 5] (2.2404000575005023,1.3552570323988937) node {$M_{34}$};
\draw [fill=neonPurple] (-0.3456550192657968,1.9101779492207251) circle (1pt);
\draw[color=neonPurple] (-0.23186039051753055,2.179343848404905) node {$E=P_{\frac{1}{2}}^4$};
\draw [fill=neonPurple] (-0.8516188026609242,0.4980870704553072) circle (1pt);
\draw[color=neonPurple, yshift = -10] (-0.7134695687028617,0.3706338236644392) node {$P_{\frac{1}{2},4}^{3}$};
\draw [fill=neonPurple] (0.5790391154745993,0.7291018478721661) circle (1pt);
\draw[color=neonPurple] (0.3246657709410742,0.6167896258480529) node {$P_{\frac{1}{2},1}^{3}$};
\draw [fill=neonPurple] (1.0696736768534196,2.4070126118012256) circle (1pt);
\draw[color=neonPurple, yshift = -13] (1.1594550131289814,2.6181433218626506) node {$P_{\frac{1}{2},2}^{3}$};
\draw [fill=neonPurple] (-1.834059047464485,2.096332317533475) circle (1pt);
\draw[color=neonPurple, yshift = -12] (-1.6552830727097312,2.1151292913135276) node {$P_{\frac{1}{2},3}^{3}$};
\end{tikzpicture}
    \end{center}
    \caption{The Euler point of a cyclic quadrilateral is a $P_{\lambda}$ point.}
\end{figure}
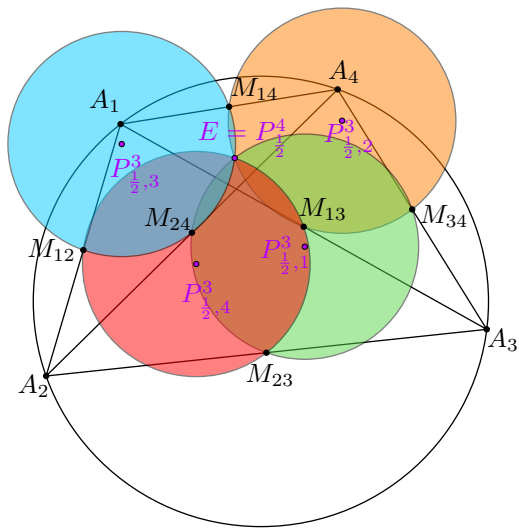

Theorem \ref{themG} thus generalizes many results in the plane, such as those discussed by Gómez \cite{Gomez} and Collings \cite{Collings}, as well as more classical results in \cite{Yaglom,Andrica}. More specifically, the well-known Euler point of a cyclic quadrilateral\footnote{Also known as the Poncelet point, quasi-orthocenter, or anticenter of a cyclic quadrilateral; for a given cyclic quadrilateral $ABCD$, the nine-point circles of $\triangle ABC, \triangle BCD, \triangle ACD, \triangle ABD$ all pass through this point.} (see Fig. 2) is a direct corollary of Theorem \ref{themG} for the case $(\lambda, n, d) = (\frac{1}{2}, 4, 2)$, since the nine-point center of a triangle $\triangle ABC$ is given by $\frac{1}{2}(a+b+c)$; or analogs of this Theorem for the centroid or orthocenter \cite{Bradley,Mammana}. Similarly, if we let $H_i$ denote the orthocenter of triangle $A_1A_2A_3A_4 \setminus A_i$ for cyclic quadrilateral $A_1A_2A_3A_4$ then lines $A_iH_i$ intersect at the Euler point of $A_1A_2A_3A_4$; this comes about as a corollary of Theorem \ref{themG} for $(\lambda, n, d) = (1, 4, 2)$. 

Importantly, the statement of Theorem \ref{themG} does not require the points $A_1, A_2, \ldots, A_n$ to be distinct or to form a non-degenerate structure; it holds for any set of points on a sphere. This observation implies that one can simply choose the $n$ points to lie on the $(d-2)$-dimensional unit-sphere. In this way, the Euler line in a lower dimension naturally arises as the "shadow" of a higher-dimensional structure, as for example, our Theorem for a tetrahedron degenerated into a cyclic quadrilateral gives rise to the above stated classical results.

\section{In the world of triangle centers}

The Euler line, as defined for a triangle, not only contains the four classical triangle centers -- the orthocenter, centroid, circumcenter, and nine-point center -- but, in the context of the Encyclopedia of Triangle Centers, it passes through $6683$ \cite{Kimberling} (as of May $4$th, 2024). We let $X_k$ denote the $k$th cataloged center in the Encyclopedia of Triangle Centers (henceforth, ETC).

We now identify which triangle centers \( X_k \) are \( P_{\lambda} \) points. For this, we use the Shinagawa coefficients, which are defined in the ETC \cite{Kimberling} for centers on the Euler line as follows.

\begin{defn}[Shinagawa Coefficients]
    For a triangle \(\triangle ABC\), let \(X\) be a triangle center with barycentric coordinates \( (f(a,b,c):f(b,c,a):f(c,a,b))\)\footnote{Here $f(a,b,c)$ is a triangle center function, as defined at the beginning of \cite{Kimberling}.}. The Shinagawa coefficients of \(X\) is the pair of functions (\(\psi(a,b,c)\), \(\delta(a,b,c)\)) such that the following condition holds cyclically:
    \begin{equation} 
    f(a,b,c) = \psi(a,b,c) \cdot S^2 + \delta(a,b,c) \cdot S_BS_C,
    \end{equation}
    where \(S_A, S_B, S_C, S\) are the Conway symbols for \(\triangle ABC\). The side lengths are \(a,b,c\), \(S\) is twice the area of \(\triangle ABC\), \(S_A = \frac{1}{2} (b^2+c^2-a^2)\), and \(S_B, S_C\) are defined cyclically.    
\end{defn}

This definition provides a condition for a triangle center to be a $P_{\lambda}$ point: its Shinagawa coefficients, \(\psi(a,b,c)\) and \(\delta(a,b,c)\), must be constants. A list of such centers is available in \cite{TablesKimberling} under the name "S-Coefficients". 

Some notable centers with constant coefficients include: $X_2 = G$ (the centroid), $X_{3} = O$ (the circumcenter), $X_4 = H$ (the orthocenter), $X_5 = N$ (the nine-point center), \(X_{20} = L\) (de Longchamps point), $X_{30}$ (the Euler infinity point), and \(X_{140}\) (the midpoint of \(X_3\) and \(X_5\)).

We now aim to determine how to convert the Shinagawa coefficients into the \(\lambda\) values of a $P_{\lambda}$ point, first giving a Lemma stated by Kimberling \cite{Kimberling}. 

\begin{lem}\label{shinagawa}
    Let a point \(X\) have Shinagawa coefficients \((u, v)\), where \(u\) and \(v\) are real numbers, that is, \(\psi(a,b,c)\) and \(\delta(a,b,c)\) are constants. Then the vector between \(X\) and \(G=X_2\), where $O=X_3=(0,0)$, is given by
    \begin{equation}
        \overrightarrow{GX} = \frac{2v}{3u + v}\overrightarrow{OG}.
    \end{equation}
\end{lem}

\begin{prop}[Conversion of Shinagawa coefficients]
    Let a center \(X\) have constant Shinagawa coefficients \((u, v)\). Then the lambda coefficient $\lambda$ associated with \(X\) is described by the function \(\tau: \mathbb{R}^2 \setminus \{ (0,0)\} \rightarrow \mathbb{R} \cup \{ \infty\}\) such that\footnote{The case where $3u+v=0$ corresponds to $\lambda = \infty$, which is the Euler infinity point, $X_{30}$, with Shinagawa coefficients $(1, -3)$.} 
    \begin{equation}\label{shinagawafor}
        \tau(u, v) = \frac{u + v}{3u + v} = \lambda.
    \end{equation}
\end{prop}

\begin{proof}
    We let $x = \lambda (a+b+c)$. By Lemma \ref{shinagawa} we get \begin{equation}\label{part1}
        \begin{split}
        \overrightarrow{GX} = \frac{2v}{3u+v} \overrightarrow{OG}  = \frac{\frac23v }{3u+v} (a+b+c).
    \end{split} 
    \end{equation} 

Furthermore, \begin{equation}\label{part2}
    \begin{split}
    \overrightarrow{GX} = x-g =  \lambda(a+b+c) - \frac13 (a+b+c) = \left(\lambda - \frac{1}{3}\right) \left(a+b+c\right).
\end{split} 
\end{equation}

Hence, by \eqref{part1} and \eqref{part2}, 
\begin{equation}
    \begin{split}
        \left(\lambda - \frac{1}{3}\right) \left(a+b+c\right) &= \frac{2v }{3u+v}\cdot \frac13 (a+b+c) \\
    \implies \tau(u,v) &= \lambda = \frac{\frac{2}{3}v }{3u+v}  + \frac{1}{3} = \frac{u+v}{3u+v},
\end{split}   
\end{equation} which concludes our proof. 
\end{proof}

According to \cite{Kimberling,TablesKimberling}, there are 721 centers $X_k$ (for $k \leq 62171$) with constant Shinagawa coefficients (as of December $30$th, 2024). We note a disconnect between the study of these centers as individual points and the understanding of their role within the broader framework presented in this paper. For instance, some specific cases of points one and three of Theorem \ref{themG} for $d=2$ and $n=4$ have been proven for particular triangle centers by Rabinowitz and Suppa \cite{Rabinowitz2025}, but our work shows these are not isolated results but consequences of a more general principle. 

\noindent \textbf{Acknowledgments.} The author would like to express his sincere gratitude to Michał Kotowski and  Adam Dzedzej for their invaluable support. Further, thank you is extended to Clark Kimberling, Radoslaw Żak and Ewelina Szajdziuk.

\end{document}